\documentclass[preprint,12pt]{elsarticle}

\usepackage{graphicx}
\usepackage[cp1251]{inputenc}
\usepackage{amssymb}
\usepackage{amsthm}
\usepackage{cmap}

\usepackage{amsmath}

\biboptions{sort&compress}

\newcommand{\sysn}{\left\{\begin{array}{rcl}}
\newcommand{\sysk}{\end{array}\right.}

\newtheorem{theorem}{Theorem}[section]
\newtheorem{lemma}[theorem]{Lemma}

\theoremstyle{example}

\newtheorem{proposition}[theorem]{Proposition}
\theoremstyle{definition}
\newtheorem{definition}[theorem]{Definition}

\journal{...}

\begin{document}

\title{Baire property of some function spaces}

\author{Alexander V. Osipov}

\address{Krasovskii Institute of Mathematics and Mechanics, \\ Ural Federal
 University, Ural State University of Economics, Yekaterinburg, Russia}

\ead{OAB@list.ru}

\author{Evgenii G. Pytkeev}

\address{Krasovskii Institute of Mathematics and Mechanics, Yekaterinburg, Russia}

\ead{pyt@imm.uran.ru}

\begin{abstract} A compact space $X$ is called {\it
$\pi$-monolithic} if for any surjective continuous mapping
$f:X\rightarrow K$ where $K$ is a metrizable compact space there
exists a metrizable compact space $T\subseteq X$ such that
$f(T)=K$. A topological space $X$ is {\it Baire} if the
intersection of any sequence of open dense subsets of $X$ is dense
in $X$. Let $C_p(X, Y)$ denote the space of all continuous
$Y$-valued functions $C(X,Y)$ on a Tychonoff space $X$ with the
topology of pointwise convergence. In this paper we have proved
that for a totally disconnected space $X$ the space
$C_p(X,\{0,1\})$ is Baire if, and only if, $C_p(X,K)$ is Baire for
every $\pi$-monolithic compact space $K$.

For a Tychonoff space $X$ the space $C_p(X,\mathbb{R})$ is Baire
if, and only if, $C_p(X,L)$ is Baire for each Fr\'{e}chet space
$L$.

We construct a totally disconnected Tychonoff space $T$ such that
$C_p(T,M)$ is Baire for a separable metric space $M$ if, and only
if, $M$ is a Peano continuum. Moreover, $C_p(T,[0,1])$ is Baire
but $C_p(T,\{0,1\})$ is not.

\end{abstract}


\maketitle 


\section{Introduction}

A topological space $X$ is {\it Baire} if the Baire Category
Theorem holds for $X$, i.e., the intersection of any sequence of
open dense subsets of $X$ is dense in $X$. A space is {\it meager}
(or {\it of the first Baire category}) if it can be written as a
countable union of closed sets with empty interior. Clearly, if
$X$ is a Baire space, then $X$ is not meager. The opposite
implication is in general not true. However, it holds for every
homogeneous space $X$ (see Theorem 2.3 in \protect\cite{LM}).
 Being a Baire space is an important topological
property for a space and it is therefore natural to ask when
function spaces are Baire. The Baire property for continuous
mappings was first considered in \protect\cite{Vid}. Then a paper
\protect\cite{LM} appeared, where various aspects of this topic
were considered. In \protect\cite{LM}, necessary and, in some
cases, sufficient conditions on a space $X$ were obtained under
which the space $C_p(X,\mathbb{R})$ is Baire.

In general, it is not an easy task to characterize when a function
space has the Baire property. The problem for $C_p(X,\mathbb{R})$
was solved independently by Pytkeev \cite{pyt1}, Tkachuk
\cite{tk,Tk} and van Douwen \protect\cite{vD}, and for a space
$B_1(X,\mathbb{R})$ of Baire-one functions  was solved by Osipov
\cite{Os}.

In this paper we present characterizations for the spaces
$C_p(X,\{0,1\})$ and $C_p(X,\mathbb{N})$ to be Baire. Moreover, we
have proved that for a totally disconnected space $X$ the space
$C_p(X,\{0,1\})$ is Baire if, and only if, $C_p(X,K)$ is Baire for
every $\pi$-monolithic compact space $K$. We also establish that,
for a Tychonoff space $X$, the space $C_p(X,\mathbb{R})$ is Baire
if, and only if, $C_p(X,L)$ is Baire for each Fr\'{e}chet space
$L$.

We construct a totally disconnected Tychonoff space $T$ such that
$C_p(T,M)$ is Baire for a separable metric space $M$ if, and only
if, $M$ is a Peano continuum. Moreover, $C_p(T,[0,1])$ is Baire
but $C_p(T,\{0,1\})$ is not.

\section{Main definitions and notation}

Throughout this paper,  all spaces are assumed to be Tychonoff,
i.e., completely regular $T_1$-spaces.
 The set of positive integers is denoted by $\mathbb{N}$ and
$\omega=\mathbb{N}\cup \{0\}$. Let $\mathbb{R}$ be the real line,
we put $\mathbb{D}=\{0,1\}\subset\mathbb{I}=[0,1]\subset
\mathbb{R}$ and let $\mathbb{Q}$ be the rational numbers. Let
$f:X\rightarrow \mathbb{R}$ be a real-valued function, then $|| f
||= \sup \{|f(x)|: x\in X\}$. Let $V=\{f\in \mathbb{R}^X:
f(x_i)\in V_i, i=1,...,n\}$ where $x_i\in X$, $V_i\subseteq
\mathbb{R}$ are bounded intervals for $i=1,...,n$, then $supp
V=\{x_1,...,x_n\}$ , $diam V=\max \{diam V_i : 1\leq i \leq n \}$.

Let $C_p(X,Y)$ be the space of continuous functions from $X$ to
$Y$ with the topology of pointwise convergence.
$C_p(X,\mathbb{R})$ is usually denoted by $C_p(X)$.

A mapping $f: X\rightarrow Y$ is said to be {\it irreducible} if
the only closed subset $Z$ of $X$, with $f(Z) = Y$ , is $Z = X$.

Recall that two sets $A$ and $B$ are {\it functionally separated}
if there exists a continuous real-valued function $f$ on $X$ such
that $f[A]\subseteq \{0\}$ and $f[B]\subseteq\{1\}$.

Let $\{X_{\lambda} : \lambda\in \Lambda \}$ be a family of
topological spaces. Let $X =\prod_{\lambda\in \Lambda}
X_{\lambda}$ be the Cartesian product with the Tychonoff topology.
Take a point $p = (p_{\lambda})_{\lambda\in \Lambda}\in X$. For
each $x=(x_{\lambda})_{\lambda\in \Lambda}\in X$, let $Supp(x) =
\{\lambda\in \Lambda : x_{\lambda}\neq p_{\lambda}\}$. Then the
subspace $\sum(p) =\{x\in X : Supp(x)$ is countable $\}$ of $X$ is
called a {\it $\sum$-product} of $\{X_{\lambda} : \lambda\in
\Lambda \}$ about $p$ ($p$ is called the {\it base point}).

A topological space $X$ is a {\it condensation} of a topological
space $Y$ if there is a continuous bijection from $Y$ onto $X$.

Recall that a {\it Peano continuum}  is called a connected and
locally connected metric compact space.

For the terms and symbols not defined here, please, consult
\cite{Arh,Eng}.

\section{Main results}

Recall that a space $X$ is {\it totally disconnected} if the
quasi-component of any point $x\in X$ consists of the point $x$
alone. Note that $X$ is totally disconnected if, and only if,
$C_p(X,\mathbb{D})$ is dense in  $\mathbb{D}^X$. Moreover,
$C_p(X,Y)$ is dense in $Y^X$ for any space $Y$ and a totally
disconnected space $X$.

\begin{definition} We say that a space $X$ has the
{\it $\mathbb{B}_{\mathbb{D}}$-property} if for every pairwise
disjoint family $\{\Delta_n: n\in \mathbb{N}\}$ of non-empty
finite sets $\Delta_n\subseteq X$, $\Delta_n=A_n\cup B_n$,
$A_n\cap B_n=\emptyset$, $n\in \mathbb{N}$, there exists a
subsequence $\{\Delta_{n_k}: k\in\mathbb{N}\}$ such that
$\bigcup\{A_{n_k}: k\in\mathbb{N}\}$ and $\bigcup\{B_{n_k}:
k\in\mathbb{N}\}$ are separated by a clopen set, i.e., there is a
clopen set $W$ such that $\bigcup\{A_{n_k}:
k\in\mathbb{N}\}\subseteq W\subseteq X\setminus \bigcup\{B_{n_k}:
k\in\mathbb{N}\}$.
\end{definition}

\begin{theorem}\label{th11} Let $X$ be a totally disconnected space. Then the following assertions are equivalent:

$(a)$ $C_p(X,\mathbb{D})$ is non-meager;

$(b)$ $C_p(X,\mathbb{D})$ is Baire;

$(c)$ $X$ has the $\mathbb{B}_{\mathbb{D}}$-property.

\end{theorem}

\begin{definition} We say that a space $X$ has the
{\it $\mathbb{B}_{\mathbb{N}}$-property} if for every pairwise
disjoint family $\{\Delta_n: n\in \mathbb{N}\}$ of non-empty
finite sets $\Delta_n\subseteq X$, $n\in \mathbb{N}$, there exists
a subsequence $\{\Delta_{n_k}: k\in\mathbb{N}\}$ and a discrete
family $\{W(\Delta_{n_k}): W(\Delta_{n_k})\supseteq \Delta_{n_k},
k\in \mathbb{N}\}$ of clopen subsets of $X$.
\end{definition}

\begin{theorem}\label{th12} Let $X$ be a totally disconnected space. Then the following assertions are equivalent:

$(a)$ $C_p(X,\mathbb{N})$ is non-meager;

$(b)$ $C_p(X,\mathbb{N})$ is Baire;

$(c)$ $X$ has the $\mathbb{B}_{\mathbb{N}}$-property.

\end{theorem}

The proofs of Theorems \ref{th11} and \ref{th12} is similar to the
proof of Theorem 3.8 in \cite{ospyt} and is therefore omitted.

\begin{lemma}(\cite{LM})\label{lem4} Let $C_p(X,\prod_{\alpha\in A}M_{\alpha})$ be
non-meager (Baire). Then $C_p(X,M_{\alpha})$ is non-meager (Baire)
for every $\alpha\in A$.
\end{lemma}

\medskip

Recall that a mapping $\varphi: K\rightarrow M$ is called {\it
almost open} if $Int \varphi(V)\neq \emptyset$ for any non-empty
open subset $V$ of $K$. Note that irreducible mappings defined on
compact spaces are almost open. Also note that if
$\varphi_{\alpha}:K_{\alpha}\rightarrow M_{\alpha}$ ($\alpha\in
A$) are surjective almost open mappings then the product mapping
$\prod\limits_{\alpha\in A}\varphi_{\alpha}:
\prod\limits_{\alpha\in A} K_{\alpha}\rightarrow
\prod\limits_{\alpha\in A} M_{\alpha}$ is also almost open.

\medskip

\begin{lemma}(Lemma 3.14 in \cite{ospyt})\label{lem6} Let $\psi: K\rightarrow M$ be a surjective continuous almost open
mapping, and let $C_p(X,K)$ be a non-meager (Baire) dense subspace
of $K^X$. Then $C_p(X,M)$ is non-meager (Baire).
\end{lemma}

\begin{lemma}(Lemma 3.15 in \cite{ospyt})\label{lem5} Let $\psi:P\rightarrow L$ be a surjective continuous almost open
mapping and let $E\subseteq P$ be a dense non-meager (Baire)
subspace in $P$. Then $\psi(E)$ is non-meager (Baire).
\end{lemma}

\begin{lemma}\label{lem7} (\cite{arhPon}) Let $L$ be a metrizable compact space without
isolated points and let $T$ be the Cantor set. Then there exists a
surjective continuous irreducible mapping $\psi:T\rightarrow L$.
\end{lemma}

\begin{theorem} Let $X$ be a totally disconnected space. Then the following assertions are equivalent:

$(a)$ $C_p(X,\mathbb{D})$ is Baire;

$(b)$ $C_p(X,K)$ is Baire for some disconnected metrizable compact
space $K$;

$(c)$ $C_p(X,L)$ is Baire for any metrizable compact space $L$.

\end{theorem}

\begin{proof} $(b)\Rightarrow(a)$. Let $K=K_1\oplus K_0$ where
$K_i\not=\emptyset$, $K_i$ is clopen for $i=0,1$. Then $\varphi:
K\rightarrow \mathbb{D}$ such that $\varphi(x)=i$ for $x\in K_i$
$(i=0,1)$ is a continuous open surjection. By Lemma \ref{lem5},
$C_p(X,\mathbb{D})$ is Baire.

$(a)\Rightarrow(c)$. By Lemma \ref{lem4},
$C_p(X,\mathbb{D}^{\omega})=(C_p(X,\mathbb{D}))^{\omega}$ is
Baire. Let $L$ be a metrizable compact space and let $L_1$ be a
metrizable compact space without isolated points. Then, by Lemma
\ref{lem7}, there exists a surjective continuous irreducible
mapping $\psi:K\rightarrow L\times L_1$ from the Cantor set $K$
onto $L\times L_1$. Then $\pi_{\alpha}\circ\psi: K\rightarrow L$
is an almost open continuous mapping. By Lemma \ref{lem6},
$C_p(X,L)$ is Baire.

$(c)\Rightarrow(b)$. It is trivial.

\end{proof}

Thus, if $C_p(X,\mathbb{D})$ is Baire for a totally disconnected
space $X$ then $C_p(X,K)$ is Baire for any metrizable space $K$.

A natural question arises: to characterize the class $\mathcal{K}$
of compact spaces that $K\in \mathcal{K}$ provided that the
condition {\it " $C_p(X,\mathbb{D})$ is Baire "} implies the
condition {\it " $C_p(X,K)$ is Baire " }.

\begin{definition} A compact space $X$ is called {\it
$\pi$-monolithic} if for any surjective continuous mapping
$f:X\rightarrow K$ where $K$ is a compact metrizable space there
exists a compact metrizable subspace $T\subseteq X$ such that
$f(T)=K$.
\end{definition}

The class of $\pi$-monolithic compact spaces is multiplicative and
it closed under continuous Hausdorff images. Hence, this class
contains monolithic and dyadic compact spaces.

\begin{lemma} A compact space $X$ is $\pi$-monolithic if, and only
if, for any sequence $\{U_n: n\in\omega\}$ of non-empty open sets
of $X$ there exists a compact metrizable subspace $K\subseteq X$
such that $K\cap U_n\neq\emptyset$ for every $n\in\omega$.
\end{lemma}

\begin{proof} Let $X$ be a $\pi$-monolithic compact space and let $\{U_n: n\in\omega\}$ be a sequence of non-empty open
sets of $X$. Choose $V_n\subseteq U_n$ such that $V_n\neq
\emptyset$, $V_n$ is a co-zero set of $X$ for every $n\in \omega$.
Then, by \cite{arhPon}, there are a surjective continuous mapping
$f:X\rightarrow Y$ where $Y$ is a compact metrizable space and
open sets $W_n\subseteq Y$ such that $f^{-1}(W_n)=V_n$ for each
$n\in \mathbb{N}$. Since $X$ is $\pi$-monolithic, there exists a
compact metrizable subspace $K\subseteq X$ such that $f(K)=Y$.
Then $K\cap V_n\neq \emptyset$ for each $n\in\omega$.

Suppose that for any sequence $\{U_n: n\in\omega\}$ of non-empty
open sets of $X$ there exists a compact metrizable subspace
$K\subseteq X$ such that $K\cap U_n\neq\emptyset$ for every
$n\in\omega$. Let $f:X\rightarrow Y$ be a continuous mapping $X$
onto a compact metrizable space $Y$. Choose a countable base
$\{V_n: n\in\omega\}$ in $Y$. Then there exists a compact
metrizable subspace $T\subseteq X$ such that $T\cap
f^{-1}(V_n)\neq\emptyset$ for each $n\in\omega$. Clear that
$f(T)=Y$.

\end{proof}

\begin{theorem} The class
$\mathcal{K}$ coincide with the class of $\pi$-monolithic compact
spaces.
\end{theorem}

\begin{proof} Let $K\in\mathcal{K}$ and $|K|\geq \omega$.

($\star$). Suppose that $X$ is an infinite totally disconnected
space, $C_p(X,Y)$ is Baire for some space $Y$ and $\{V_m: m\in
\omega\}$ is a sequence of non-empty open sets of $Y$. Then, there
is $f\in C(X,Y)$ such that $f(X)\cap V_m\neq \emptyset$ for each
$m\in \omega$.

Choose a disjoint family $\{\Delta_n: n\in\omega\}$ of finite
subsets of $X$ such that $|\Delta_n|=n$ for $n\in\omega$. Let
$M_n=\{f: f(\Delta_n)\cap V_i\neq\emptyset$, $i\leq n\}$.
$F_m=\bigcap\limits_{n\geq m}(C(X,Y)\setminus M_n)$, $n,m\in
\omega$. Since $X$ is totally disconnected, $F_m$ is nowhere dense
in $C(X,Y)$ for each $m\in\omega$. Hence there is $f\in
C(X,Y)\setminus \bigcup\limits_{m\in\omega} F_m$. Then $f(X)\cap
V_i\neq \emptyset$ for each $i\in\omega$.

By replacing $\mathbb{I}$ with $\mathbb{D}$ in \cite{Sh}, we
construct $X_0$ such that

(1) $X_0\subseteq \mathbb{D}^\mathfrak{c}$ and
$\beta(X)=\mathbb{D}^\mathfrak{c}$;

(2) every countable subset $S$ of $X_0$ is closed and
$\mathbb{D}$-embedded.

By the condition (2) and Theorem \ref{th11}, $C_p(X_0,\mathbb{D})$
is Baire. Let $\{U_n: n\in\omega\}$ be a sequence of non-empty
open subsets in $K$. By ($\star$), there is $f\in C(X_0,K)$ such
that $f(X_0)\cap U_i\neq \emptyset$ for $i\in\omega$. Then, by
(1), there is a continuous extension
$\widetilde{f}:\mathbb{D}^\mathfrak{c}\rightarrow K$ of $f$. Then,
$\widetilde{f}(\mathbb{D}^\mathfrak{c})\cap U_i\neq\emptyset$ for
$i\in\omega$. Since $\widetilde{f}(\mathbb{D}^\mathfrak{c})$ is a
dyadic compact space (hence, it is $\pi$-monolithic), there is a
compact metrizable subspace $\widetilde{K}\subseteq
\widetilde{f}(\mathbb{D}^\mathfrak{c})$ such that
$\widetilde{K}\cap (\widetilde{f}(\mathbb{D}^\mathfrak{c})\cap
U_i)\neq\emptyset$ for $i\in\omega$. It follows that $K$ is
$\pi$-monolithic.

Suppose that $K$ is $\pi$-monolithic, $X$ is totally disconnected
and $C_p(X,\mathbb{D})$ is Baire. We claim that $C_p(X,K)$ is
Baire.

Assume the contrary. Then, there is a basis open set $P=\bigcap
\{M(x_i,V_i): i\leq k_0\}$ such that $P$ is meager, i.e.,
$P=\bigcup\limits_{m\in\omega} F_m$ where $F_m$ is nowhere dense
and $F_m\subseteq F_{m+1}$ for each $m\in\omega$. Let
$\Delta_0=\{x_i:i\leq k_0\}$, $\mu_0=\{V_i: i\leq k_0\}$.
Construct, by induction, finite disjoint sets $\Delta_n=\{x_{n_i},
i\leq k_n\}\subseteq X$, $\Delta_n\cap \Delta_0=\emptyset$,
$n\in\omega$, finite families open (in $K$) sets $\mu_n$,
$\mu_n\subseteq \mu_{n+1}$, $n\in\omega$, and open (in $K$) sets
$V_{n_i}\in \mu_n$, $i\leq k_n$, $n\geq 2$ such that

(3) for every basis set $\bigcap \{M(x,V(x)):x\in
\bigcup\limits_{i=0}^n \Delta_i\}$, $n\geq 1$, where $V(x)\in
\mu_n$, $V(x_i)\subseteq V_i$, $i\leq n_0$ there are $V'(x)\in
\mu_{n+1}$, $V'(x)\subseteq V(x)$, $x\in \bigcup\limits_{i=0}^n
\Delta_i$ such that $\bigcap \{M(x,V'(x)): x\in
\bigcup\limits_{i=0}^n \Delta_i\}\cap \bigcap
\{M(x_{{n+1},i},V_{n+1,i}): i\leq k_{n+1}\}\cap F_n=\emptyset$,

(4) for every $V\in\mu_n$ there is $V'\in \mu_{n+1}$ such that
$\overline{V'}\subseteq V$, $n\geq 0$.

Since $K$ is $\pi$-monolithic, there is a compact metrizable
subspace $\widetilde{K}$ such that $\widetilde{K}\cap V\neq
\emptyset$ for every $V\in \bigcup\limits_{n=0}^{\infty} \mu_n$.

We claim that there is a non-empty compact metrizable subspace
$K^*\subseteq \widetilde{K}$ such that for every $V\in
\bigcup\limits_{n=0}^{\infty} \mu_n$, $V\cap K^*\neq \emptyset$
and $\{V\cap K^*: V\in \bigcup\limits_{n=0}^{\infty} \mu_n\}$ is a
$\pi$-base of $K^*$.

Let $R=\{F: F$ is a non-empty closed subset of $\widetilde{K}$
such that there is a sequence $V_m\in
\bigcup\limits_{n=0}^{\infty} \mu_n$, $\overline{V_{m+1}}\subseteq
V_m$, $m\in\omega$, $\{V_m\cap \widetilde{K}: m\in\omega\}$ is a
base of $F\}$. Let $K^*=\bigcap\{ S: S$ is closed subset of
$\widetilde{K}$ such that $S\cap F\neq \emptyset$ for every $F\in
R\}$. Then $\{F\cap K^*: F\in R\}$ is a $\pi$-network of $K^*$. By
(4), $V\cap K^*\neq \emptyset$ for every $V\in
\bigcup\limits_{n=0}^{\infty} \mu_n$ and $\{V\cap K^*: V\in
\bigcup\limits_{n=0}^{\infty} \mu_n\}$ is a $\pi$-base of $K^*$.
Let $M_{p+1}=\bigcup \{\bigcap\{M(x,V(x)\cap K^*): x\in
\bigcup\limits_{i=0}^{n+1} \mu_i$ where $V(x_i)\subseteq V_i$,
$i\leq k_0$, $n\geq p+1$, $V(x_{n+1,i})\subseteq V_{n+1,i}$,
$i\leq k_{n+1}$, $V(x)\in \bigcup\limits_{i=0}^{n+1} \mu_i$ and
$\bigcap \{M(x,V(x)): x\in \bigcup\limits_{i=0}^{n+1} \mu_i\}\cap
F_n=\emptyset\}$, $p\in \omega$. Then $M_{n+1}$ is a non-empty
open subset of $C_p(X,K^*)$, $n\geq 0$.

We claim that $M_{p+1}$ is dense in $P^*=\bigcap \{M(x_i, V_i\cap
K^*): i\leq k_0\}\subseteq C_p(X,K^*)$. Let $\varphi\in P^*$ and
let $O(\varphi)$ be a base neighborhood of $\varphi$ in
$C(X,K^*)$. We can assume that $O(\varphi)=\bigcap \{M(x,W(x)):
x\in \bigcup\limits_{i=0}^{m} \mu_i\}\cap\bigcap\{M(y,W(y)):y\in
T\subseteq X\setminus \bigcup\limits_{i=0}^{\infty} \Delta_i\}$
where $W(x_i)\subseteq V_i\cap K^*$, $i\leq k_0$ and $W(x)$ (
$x\in \bigcup\limits_{i=0}^{m} \Delta_i$), $W(y)$ ($y\in T$) are
non-empty open sets in $K^*$ , $T$ is finite and $m\in\omega$.

Since $\{V\cap K^*: V\in\bigcup\limits_{n=0}^{\infty} \mu_n\}$ is
a $\pi$-base of $K^*$, there are $V(x)\in
\bigcup\limits_{n=0}^{\infty} \mu_n$ for $x\in
\bigcup\limits_{i=0}^{m} \Delta_i$ such that $V(x)\cap
K^*\subseteq W(x)$ for $x\in \bigcup\limits_{i=0}^{m} \Delta_i$.
Then, there is $k\in\omega$ such that $V(x)\in \mu_l$ for
$x\in\bigcup\limits_{i=0}^{m} \Delta_i$, $k\geq m$, $k\geq p+1$.
By (3), there are sets $V'(x)\in \mu_{l+1}$,
$x\in\bigcup\limits_{i=0}^{l+1} \Delta_i$ such that
$V'(x)\subseteq V(x)$, $x\in \bigcup\limits_{i=0}^{m} \Delta_i$,
$V'(x_{l+1,i})\subseteq V_{l+1,i}$, $i\leq k_{l+1}$ and $\bigcap
\{M(x,V'(x)): x\in\bigcup\limits_{i=0}^{k+1} \Delta_i\}\cap
F_k=\emptyset$.

Choose $g\in \bigcap \{M(x,V'(x)\cap K^*):
x\in\bigcup\limits_{i=0}^{k+1} \Delta_i\}\cap \bigcap \{M(y,W(y)):
y\in T\}$. Then $g\in O(\varphi)\cap P^*\cap M_{p+1}$. Hence,
$M_{p+1}$ is a dense open set in $P^*$. By Theorem \ref{th11},
$C_p(X,K^*)$ is Baire, hence, $\bigcap\limits_{p=0}^{\infty}
M_{p+1}\neq \emptyset$. Let $g\in \bigcap\limits_{p=0}^{\infty}
M_{p+1}$. So we proved that $g\not\in
\bigcup\limits_{m=1}^{\infty} F_m=P$ which is a contradiction.

\end{proof}

\begin{lemma}\cite{arhPon}\label{lem10} For any Polish space $L$ there is a
surjective continuous open mapping $\varphi:
\omega^{\omega}\rightarrow L$.
\end{lemma}

\begin{theorem} Let $X$ be a totally disconnected space. Then the following assertions are equivalent:

$(a)$ $C_p(X,\mathbb{N})$ is Baire;

$(b)$ $C_p(X,M)$ is Baire for some complete metric non-compact
zero-dimensional space $M$;

$(c)$ $C_p(X,K)$ is Baire for every complete metric space $K$.

\end{theorem}

\begin{proof} $(b)\Rightarrow(a)$. Since $M$ is a non-compact
zero-dimensional metric space, $M=\bigoplus \{M_i:i\in\omega\}$
where $M_i$ is a non-empty clopen set for every $i\in\omega$.
Consider $\psi: M\rightarrow \mathbb{N}$ such that $\psi(M_i)=i$
for each $i\in\omega$. Then $\psi$ is a continuous open mapping.
By Lemma \ref{lem6}, $C_p(X,\mathbb{N})$ is Baire.

$(a)\Rightarrow(c)$. Note that, by Lemmas \ref{lem4} and
\ref{lem10}, the implication is true for a Polish space $K$.

Let $K$ be a complete metric space with a complete metric $\rho$.

Assume the contrary. Then, there is a basis open set $P=\bigcap
\{M(x_i,V_i): i\leq k_0\}$ such that $P$ is meager, i.e.,
$P=\bigcup\limits_{m\in\omega} F_m$ where $F_m$ is nowhere dense
and $F_m\subseteq F_{m+1}$ for each $m\in\omega$. Let
$\Delta_0=\{x_i:i\leq k_0\}$, $\mu_0=\{V_i: i\leq k_0\}$.
Construct, by induction, finite disjoint sets $\Delta_n=\{x_{n_i},
i\leq k_n\}\subseteq X$, $\Delta_n\cap \Delta_0=\emptyset$,
$n\in\omega$, finite families open (in $K$) sets $\mu_n$,
$\mu_n\subseteq \mu_{n+1}$, $n\in\omega$, and open (in $K$) sets
$V_{n_i}\in \mu_n$, $i\leq k_n$, $n\geq 2$ such that

(1) for every basis set $\bigcap \{M(x,V(x)):x\in
\bigcup\limits_{i=0}^n \Delta_i\}$, $n\geq 1$, where $V(x)\in
\mu_n$, $V(x_i)\subseteq V_i$, $i\leq n_0$ there are $V'(x)\in
\mu_{n+1}$, $V'(x)\subseteq V(x)$, $x\in \bigcup\limits_{i=0}^n
\Delta_i$ such that $\bigcap \{M(x,V'(x)): x\in
\bigcup\limits_{i=0}^n \Delta_i\}\cap \bigcap
\{M(x_{{n+1},i},V_{n+1,i}): i\leq k_{n+1}\}\cap F_n=\emptyset$,

(2) for every $V\in\mu_n$ there is $V'\in \mu_{n+1}$ such that
$\overline{V'}\subseteq V$, $n\geq 0$;

(3) for every $V\in \mu_n$, $diam V<\frac{1}{n}$, $n\in\omega$.

Consider $\widetilde{K}=\{x\in K:$ there is a base $\{V_n(x):
V_n(x)\in \bigcup\limits_{k=0}^{\infty} \mu_k, n\in\omega\}$ of
nieghborhoods of $x\}$.

By (1) and (2), $V\cap K^*\neq\emptyset$ for every
$V\in\bigcup\limits_{k=0}^{\infty} \mu_k$ and $\{V\cap
\widetilde{K}: V\in\bigcup\limits_{k=0}^{\infty} \mu_k\}$ is a
base in $\widetilde{K}$. Then $K^*=\overline{\widetilde{K}}$ is a
Polish space and $\{V\cap K^*: V\in\bigcup\limits_{k=0}^{\infty}
\mu_k\}$ is a $\pi$-base of $K^*$.

Let $M_{p+1}=\bigcup \{\bigcap\{M(x,V(x)\cap K^*): x\in
\bigcup\limits_{i=0}^{n+1} \mu_i$ where $V(x_i)\subseteq V_i$,
$i\leq k_0$, $n\geq p+1$, $V(x_{n+1,i})\subseteq V_{n+1,i}$,
$i\leq k_{n+1}$, $V(x)\in \bigcup\limits_{i=0}^{n+1} \mu_i$ and
$\bigcap \{M(x,V(x)): x\in \bigcup\limits_{i=0}^{n+1} \mu_i\}\cap
F_n=\emptyset\}$, $p\in \omega$. Then $M_{n+1}$ is a non-empty
open subset of $C_p(X,K^*)$, $n\geq 0$.

We claim that $M_{p+1}$ is dense in $P^*=\bigcap \{M(x_i, V_i\cap
K^*): i\leq k_0\}\subseteq C_p(X,K^*)$. Let $\varphi\in P^*$ and
let $O(\varphi)$ be a base neighborhood of $\varphi$ in
$C(X,K^*)$. We can assume that $O(\varphi)=\bigcap \{M(x,W(x)):
x\in \bigcup\limits_{i=0}^{m} \mu_i\}\cap\bigcap\{M(y,W(y)):y\in
T\subseteq X\setminus \bigcup\limits_{i=0}^{\infty} \Delta_i\}$
where $W(x_i)\subseteq V_i\cap K^*$, $i\leq k_0$ and $W(x)$ (
$x\in \bigcup\limits_{i=0}^{m} \Delta_i$), $W(y)$ ($y\in T$) are
non-empty open sets in $K^*$ , $T$ is finite and $m\in\omega$.

Since $\{V\cap K^*: V\in\bigcup\limits_{n=0}^{\infty} \mu_n\}$ is
a $\pi$-base of $K^*$, there are $V(x)\in
\bigcup\limits_{n=0}^{\infty} \mu_n$ for $x\in
\bigcup\limits_{i=0}^{m} \Delta_i$ such that $V(x)\cap
K^*\subseteq W(x)$ for $x\in \bigcup\limits_{i=0}^{m} \Delta_i$.
Then, there is $k\in\omega$ such that $V(x)\in \mu_l$ for
$x\in\bigcup\limits_{i=0}^{m} \Delta_i$, $k\geq m$, $k\geq p+1$.
By (1), there are sets $V'(x)\in \mu_{l+1}$,
$x\in\bigcup\limits_{i=0}^{l+1} \Delta_i$ such that
$V'(x)\subseteq V(x)$, $x\in \bigcup\limits_{i=0}^{m} \Delta_i$,
$V'(x_{l+1,i})\subseteq V_{l+1,i}$, $i\leq k_{l+1}$ and $\bigcap
\{M(x,V'(x)): x\in\bigcup\limits_{i=0}^{k+1} \Delta_i\}\cap
F_k=\emptyset$.

Choose $g\in \bigcap \{M(x,V'(x)\cap K^*):
x\in\bigcup\limits_{i=0}^{k+1} \Delta_i\}\cap \bigcap \{M(y,W(y)):
y\in T\}$. Then $g\in O(\varphi)\cap P^*\cap M_{p+1}$. Hence,
$M_{p+1}$ is a dense open set in $P^*$. By Theorem \ref{th11},
$C_p(X,K^*)$ is Baire, hence, $\bigcap\limits_{p=0}^{\infty}
M_{p+1}\neq \emptyset$. Let $g\in \bigcap\limits_{p=0}^{\infty}
M_{p+1}$. So we proved that $g\not\in
\bigcup\limits_{m=1}^{\infty} F_m=P$ which is a contradiction.

$(c)\Rightarrow(b)$. It is trivial.
\end{proof}

Recall that a topological vector space $X$ is a {\it Fr\'{e}chet
space} if $X$ is a locally convex complete metrizable space.

\begin{theorem} Let $X$ be a Tychonoff space. Then the following assertions are equivalent:

$(a)$ $C_p(X)$ is Baire;

$(b)$ $C_p(X,M)$ is Baire for some Fr\'{e}chet space $M$, $dim
M>0$;

$(c)$ $C_p(X,L)$ is Baire for any Fr\'{e}chet space $L$.

\end{theorem}

\begin{proof} $(b)\Rightarrow(a)$.
 Let $\varphi$ be a continuous liner functional on $M$,
 $\varphi\neq 0$. By \cite{sh} (see Chapter III, Corollary 1), $\varphi:M\rightarrow \mathbb{R}$ is
 a surjective continuous open mapping. By Lemma \ref{lem6}, $C_p(X)$
 is Baire.

$(a)\Rightarrow(c)$. Assume the contrary. Then, there is a basis
open set $P=\bigcap \{M(x_i,V_i): i\leq
k_0\}=\bigcup\limits_{m=1}^{\infty} F_m$ where $F_m$ is  a nowhere
dense set and $F_m\subseteq F_{m+1}$ for all $m\in\omega$. Let
$\Delta_0=\{x_i:i\leq k_0\}$, $\mu_0=\{V_i:i\leq k_0\}$.
Construct, by induction, finite sets $\Delta_n=\{x_{n_i}: i\leq
k_n\}\subseteq X$, $\Delta_n\cap \Delta_0=\emptyset$,
$n\in\omega$, finite families $\mu_n$, $\mu_n\subseteq \mu_{n+1}$,
$n\in\omega$, of open sets in $L$, open sets $V_{n_i}\in \mu_n$ in
$L$, $i\leq k_n$, $n\geq 2$, separable linear subspaces $L_n$,
$L_n\subseteq L_{n+1}$, $n\in\omega$, families
$\mathcal{B}_n=\{O_{n,m}: m\in\omega\}$ of open sets in $L$ for
$n\in\omega$, such that

(1) for every basic set $\bigcap \{M(x,V(x)): x\in
\bigcup\limits_{i=0}^n \Delta_i\}$, $n>1$, where $V(x)\in\mu_n$,
$V(x_i)\subseteq V_i$, $i\leq n_0$, there are $V'(x)\in\mu_{n+1}$,
$V'(x)\subseteq V(x)$, $x\in \bigcup\limits_{i=0}^n \Delta_i$ such
that $\bigcap \{M(x,V'(x)): x\in \bigcup\limits_{i=0}^n
\Delta_i\}\cap \bigcap\{M(x_{n+1,i},V_{n+1,i}): i\leq
k_{n+1}\}\cap F_n=\emptyset$,

(2) for every $V\in \mu_n$ there is $V'\in \mu_{n+1}$ such that
$\overline{V'}\subseteq V$, $n\geq 0$,

(3) $\mathcal{B}_n=\{O_{n,m}:m\in\omega\}$ such that
$\{O_{n,m}\cap L_n: m\in\omega\}$ is a base of $L_n$ and $diam
O_{n,m}\rightarrow 0$ ($m\rightarrow \infty$), $n\in \omega$,

(4) $\mu_n\subseteq \mu_{n+1}$ and $\{O_{k,m}: k,m\leq n+1\}$,
$n\in \omega$.

Let $L^*=\overline{\bigcup\limits_{n=1}^{\infty} L_n}$. Then $L^*$
is a separable Fr\'{e}chet space. By (3) and (4), $V\cap L^*\neq
\emptyset$ for all $V\in \bigcup\limits_{n=0}^{\infty} \mu_n$ and
$\{V\cap L^*: V\in \bigcup\limits_{n=0}^{\infty} \mu_n\}$ is a
$\pi$-base in $L^*$. By \cite{Tur}, $L^*$ is homeomorphic to
$\mathbb{R}^{\alpha}$, $\alpha\leq \aleph_0$. Hence, $C_p(X,L^*)$
is Baire. Let $M_{p+1}=\bigcup \{\bigcap\{M(x,V(x)\cap K^*): x\in
\bigcup\limits_{i=0}^{n+1} \mu_i$ where $V(x_i)\subseteq V_i$,
$i\leq k_0$, $n\geq p+1$, $V(x_{n+1,i})\subseteq V_{n+1,i}$,
$i\leq k_{n+1}$, $V(x)\in \bigcup\limits_{i=0}^{n+1} \mu_i$ and
$\bigcap \{M(x,V(x)): x\in \bigcup\limits_{i=0}^{n+1} \mu_i\}\cap
F_n=\emptyset\}$, $p\in \omega$. Then $M_{n+1}$ is a non-empty
open subset of $C_p(X,K^*)$ for $n\geq 0$.

We claim that $M_{p+1}$ is dense in $P^*=\bigcap \{M(x_i, V_i\cap
K^*): i\leq k_0\}\subseteq C_p(X,K^*)$. Let $\varphi\in P^*$ and
let $O(\varphi)$ be a base neighborhood of $\varphi$ in
$C(X,K^*)$. We can assume that $O(\varphi)=\bigcap \{M(x,W(x)):
x\in \bigcup\limits_{i=0}^{m} \mu_i\}\cap\bigcap\{M(y,W(y)):y\in
T\subseteq X\setminus \bigcup\limits_{i=0}^{\infty} \Delta_i\}$
where $W(x_i)\subseteq V_i\cap K^*$, $i\leq k_0$ and $W(x)$ (
$x\in \bigcup\limits_{i=0}^{m} \Delta_i$), $W(y)$ ($y\in T$) are
non-empty open sets in $K^*$ , $T$ is finite and $m\in\omega$.

Since $\{V\cap K^*: V\in\bigcup\limits_{n=0}^{\infty} \mu_n\}$ is
a $\pi$-base of $K^*$ there are $V(x)\in
\bigcup\limits_{n=0}^{\infty} \mu_n$ for $x\in
\bigcup\limits_{i=0}^{m} \Delta_i$ such that $V(x)\cap
K^*\subseteq W(x)$ for $x\in \bigcup\limits_{i=0}^{m} \Delta_i$.
Then, there is $k\in\omega$ such that $V(x)\in \mu_l$ for
$x\in\bigcup\limits_{i=0}^{m} \Delta_i$, $k\geq m$, $k\geq p+1$.
By (3), there are sets $V'(x)\in \mu_{l+1}$,
$x\in\bigcup\limits_{i=0}^{l+1} \Delta_i$ such that
$V'(x)\subseteq V(x)$, $x\in \bigcup\limits_{i=0}^{m} \Delta_i$,
$V'(x_{l+1,i})\subseteq V_{l+1,i}$, $i\leq k_{l+1}$ and $\bigcap
\{M(x,V'(x)): x\in\bigcup\limits_{i=0}^{k+1} \Delta_i\}\cap
F_k=\emptyset$.

Choose $g\in \bigcap \{M(x,V'(x)\cap K^*):
x\in\bigcup\limits_{i=0}^{k+1} \Delta_i\}\cap \bigcap \{M(y,W(y)):
y\in T\}$. Then $g\in O(\varphi)\cap P^*\cap M_{p+1}$. Hence,
$M_{p+1}$ is a dense open set in $P^*$. By Theorem \ref{th11},
$C_p(X,K^*)$ is Baire, hence, $\bigcap\limits_{p=0}^{\infty}
M_{p+1}\neq \emptyset$. Let $g\in \bigcap\limits_{p=0}^{\infty}
M_{p+1}$. So we proved that $g\not\in
\bigcup\limits_{m=1}^{\infty} F_m=P$ which is a contradiction.

$(c)\Rightarrow(b)$. It is trivial.

\end{proof}

We will next give an example of a totally disconnected space $T$
such that, for any second countable space $M$, the space
$C_p(T,M)$ has the Baire property if and only if $M$ is a Peano
continuum. In particular, the space $C_p(T,\mathbb{I})$ is Baire
but $C_p(T,\mathbb{D})$ is not.

We use the following lemma in the construction of a space $T$.

\begin{lemma}\label{lem21} Let $X$ be a pseudocompact infinite space
and let $C_p(X,Y)$ be a dense set in $Y^X$. Then, if $C_p(X,Y)$ is
Baire then  $Y$ is pseudocompact.
\end{lemma}

\begin{proof} Assume the contrary. There is a unbounded continuous mapping $\varphi:
Y\rightarrow \mathbb{R}$. Then,
$C_p(X,Y)=\bigcup\limits_{m=1}^{\infty} F_m$ where $F_m=\{f\in
C(X,Y): ||\varphi\circ f||\leq m\}$, $m\in\omega$. Since
$C_p(X,Y)$ is dense in $Y^X$ and $X$ is infinite, $F_m$ is nowhere
dense which is a contradiction.
\end{proof}

\begin{proposition} There exists a totally disconnected Tychonoff
space $T$ such that $C_p(T,M)$ is Baire where $M$ is a separable
metric space if, and only if, $M$ is a Peano continuum.
\end{proposition}

\begin{proof} By \cite{Sh}, there is $Z\subseteq \mathbb{I}^A$ where
$|A|=\mathfrak{c}$ such that

$(1)$ $|Z|=\mathfrak{c}$;

$(2)$ every countable subset $S\subset Z$ is closed and
$C^*$-embedded;

$(3)$ $Z$ is $\aleph_0$-dense in $\mathbb{I}^A$.

Let $\sum({\bf 0})\subseteq \mathbb{D}^A$ be a $\sum$-product
about ${\bf 0}$ (${\bf 0}$ is a function $f$ from $A$ into
$\mathbb{D}$ such that $f(\alpha)=0$ for each $\alpha\in A$), and
let $\{A_{\alpha}\times B_{\alpha}: \alpha<\mathfrak{c}\}$ be a
base of $G_{\delta}$-topology $\sum({\bf 0})\times Z$. It is clear
that $|A_{\alpha}|=|B_{\alpha}|=\mathfrak{c}$ for
$\alpha<\mathfrak{c}$. Let $\sum({\bf
0})=\{m_{\alpha}:\alpha<\mathfrak{c}\}$ and
$Z=\{z_{\alpha}:\alpha<\mathfrak{c}\}$. By induction, we construct
sets $\{T_{\alpha}: \alpha<\mathfrak{c}\}$ such that

$(4)$ $T_{\beta}\supseteq T_{\alpha}$, $\beta>\alpha$,
$|T_{\alpha+1}\setminus T_{\alpha}|<\aleph_0$, $T_{\beta}=\bigcup
\{T_{\alpha}: \alpha<\beta\}$ for a limit ordinal $\beta$;

$(5)$ $T_{\alpha}\cap (A_{\beta}\times B_{\beta})\neq \emptyset$,
$\alpha<\mathfrak{c}$, $\beta<\alpha$;

$(6)$ $\pi_{\sum({\bf 0})}(T_{\alpha})\supseteq \{m_{\beta}:
\beta<\alpha\}$, $\pi_Z(T_{\alpha})\supseteq \{z_{\alpha}:
\beta<\alpha\}$.

Let $T_0=\{(m_0,z_0)\}$. Assume that $T_{\alpha}$, $\alpha<\beta$
 are constructed. If $\beta$ is limit then $T_{\beta}=\bigcup
\{T_{\alpha}: \alpha<\beta\}$. Let $\beta=\beta^{-}+1$. If
$m_{\beta^{-}}\not\in \pi_{\sum({\bf 0})}T_{\beta^{-}}$ then
choose $z_{\gamma}\in Z\setminus \pi_{Z}T_{\beta^{-}}$. Let
$T'_{\beta}=T_{\beta^{-}}\cup \{(m_{\beta^{-}}, z_{\gamma})\}$. If
$z_{\beta^{-}}\not\in \pi_ZT'_{\beta}$ then choose $m_{\gamma'}\in
\sum({\bf 0})\setminus \pi_{\sum({\bf 0})}T'_{\beta}$. Let
$T''_{\beta}=T'_{\beta}\cup\{(m_{\gamma'},z_{\beta^{-}})\}$.

Let $m\in A_{\beta^{-}}\setminus \pi_{\sum({\bf 0})}T''_{\beta}$,
$z\in B_{\beta^{-}}\setminus \pi_Z T''_{\beta}$. Then
$T_{\beta}=T''_{\beta}\cup\{(m,z)\}$. It is clear that the
conditions (4), (5) and (6) for $\{T_{\alpha}:\alpha\leq \beta\}$
hold.

We claim that $T=\bigcup\{T_{\alpha}: \alpha<\mathfrak{c}\}$ is as
required.

Note that

(7) by (4),(6), $\pi_{\sum({\bf 0})}\upharpoonright T:
T\rightarrow \sum({\bf 0})$ and $\pi_Z\upharpoonright T:
T\rightarrow Z$ are bijections.

By (5),

(8) $T$ is $\aleph_0$-dense in $\sum({\bf 0})\times Z$, hence, in
$\mathbb{D}^A\times \mathbb{I}^A$.

By (2), $C_p(Z,\mathbb{I})$ is Baire. Since $Z$ is a condensation
of $T$, $C_p(T,\mathbb{I})$ is Baire. By Theorem 3.18 in
\cite{ospyt}, $C_p(T,K)$ is Baire if $K$ is a Peano continuum.

Since $\sum({\bf 0})$ is a condensation of $T$, $T$ is a totally
disconnected. By (8), $T$ is pseudocompact. By Lemma \ref{lem21},
$M$ is compact.

(I). We claim that $M$ is connected. Let $U\cap V=\emptyset$ where
$U$, $V$ are open sets of $M$. Choose $B\subseteq A$,
$|B|=\aleph_0$ and a family $\{\Delta_n: n\in\omega\}$ of finite
disjoint sets of $\mathbb{D}^B$ and $\Delta_n$ is a
$\frac{1}{n}$-network in $\mathbb{D}^B$, $n\in\omega$. Let
$\Delta'_n=\Delta_n\times 0\subseteq \sum({\bf 0})$, $n\in
\omega$. Since $\pi_{\sum(\mathbb{D}}\upharpoonright T:
T\rightarrow \sum({\bf 0})$ is a condensation, there are finite
disjoint sets $\Delta''_n$ such that $\pi_{\sum({\bf
0})}\Delta''_n=\Delta'_n$. Let $M_n=\{f\in C(T,M):
f(\Delta''_{2n})\subseteq U$, $f(\Delta''_{2n+1})\subseteq V\}$
and $F_m=\bigcap\limits_{n\geq m}(C(T,M)\setminus M_n)$, $m,n\in
\omega$. Since $\{\Delta''_n:n\in\omega\}$ is disjoint and $T$ is
totally disconnected, $F_m$ is nowhere dense. Since $C_p(T,M)$ is
Baire, there is $\varphi\in C(T,M)\setminus
\bigcup\limits_{m=1}^{\infty} F_m$. It follows that there is a
sequence $\{n_k\}$, $n_{k+1}>n_k$ such that $\varphi\in M_{n_k}$,
$k\in\omega$. Since $T$ is a $\aleph_0$-dense in
$\mathbb{D}^A\times \mathbb{I}^A$, by \cite{Arh} (see Lemma
0.2.3), there are a countable set $B'\supseteq B$ and a continuous
mapping $\varphi': \mathbb{D}^{B'}\times
\mathbb{I}^{B'}\rightarrow M$ such that
$\varphi=\varphi'\pi_{\mathbb{D}^{B'}\times \mathbb{I}^{B'}}$. Let
$\pi_{\mathbb{D}^{B'}\times
\mathbb{I}^{B'}}(\Delta''_n)=\nabla_n$, $n\in\omega$. Then
$\varphi'(\nabla_{2n_k})\subseteq U$,
$\varphi'(\nabla_{2n_k+1})\subseteq V$.

Since $\pi_{\mathbb{D}^{B'}} \nabla_n=\Delta_n\times 0$ where
$0\in \mathbb{D}^{B'\setminus B}$, there are open sets
$V(\nabla_n)=W(\Delta_n)\times O_n(0)$, $n=2n_k, 2n_k+1$,
$k\in\omega$ where $W(\Delta_n)$ is open in $\mathbb{D}^B$,
$O_n(0)$ is open in $\mathbb{D}^{B'\setminus B}$ such that
$\varphi'(V(\nabla_{2n_k}))\subseteq U$,
$\varphi'(V(\nabla_{2n_k+1}))\subseteq V$, $k\in\omega$. Since
$\Delta_n$ is a $\frac{1}{n}$-network in $\mathbb{D}^B$,
$n\in\omega$,
$\bigcap\limits_{k=1}^{\infty}(\bigcup\limits_{m=k}^{\infty}
W(\Delta_{2n_m}))\cap
\bigcap\limits_{k=1}^{\infty}(\bigcup\limits_{m=k}^{\infty}W(\Delta_{2n_m+1}))\neq\emptyset$,
i.e., there are $t$, $2n_p$, $2n_l+1$ such that $t\in
W(\Delta_{2n_p})\cap W(\Delta_{2n_l+1})$. Then
$\widetilde{t}=t\times 0$ where $0\in \times
\mathbb{D}^{B'\setminus B}$ belongs $V(\nabla_{2n_p})\cap
V(\nabla_{2n_l+1})$. Hence, $\varphi'(\widetilde{t}\times
\mathbb{I}^{B'})\cap U\neq \emptyset$ and
$\varphi'(\widetilde{t}\times \mathbb{I}^{B})\cap V\neq
\emptyset$. Thus, for any non-empty open sets $U$, $V$ in $M$
there is the connected set
$S=\varphi'(\widetilde{t}\times\mathbb{I}^{B'})$ such that $S\cap
U \neq\emptyset$ and $S\cap V\neq\emptyset$. It follows that $M$
is connected.

(II). We claim that for any sequence $\{U_n: n\in\omega\}$ of
non-empty open sets in $M$ there is a Peano continuum $P$ such
that $|\{n: U_n\cap P\neq\emptyset\}|=\aleph_0$.

Let $\Delta_n\subseteq \mathbb{D}^B$, $\Delta'_n$, $\Delta''_n$ be
sequences from (I).

Let $\widetilde{M}_n=\{f\in C(T,M): f(\Delta''_n)\subseteq U_n\}$,
$F_m=\bigcap\limits_{n\geq n}(C(T,M)\setminus \widetilde{M}_n)$,
$n,m\in\omega$, $\varphi\in C(T,M)\setminus
\bigcup\limits_{m=1}^{\infty} F_m$.

Then there is a sequence $\{n_k\}$, $n_{k+1}>n_k$, $k\in\omega$,
such that $\varphi\in \widetilde{M}_{n_k}$, $k\in\omega$.

Further, as in (I), there are a countable set
$\widetilde{B}\supseteq B$,
$\widetilde{\varphi}:\mathbb{D}^{\widetilde{B}}\times
\mathbb{I}^{\widetilde{B}}\rightarrow M$, open sets
$V(\nabla_{n_k})=W(\Delta_{n_k})\times O_{n_k}(0)$, $k\in\omega$
such that $\widetilde{\varphi}(V(\nabla_{n_k}))\subseteq U_{n_k}$,
$k\in\omega$. Let $t\in \bigcap\limits_{k=1}^{\infty}
\bigcup\limits_{m=k}^{\infty} W(\Delta_{n_m})$. Then
$\widetilde{t}=t\times 0$ where $0\in
\mathbb{D}^{\widetilde{B}\setminus B}$, belongs
$V(\nabla_{n_{m_l}})$ for a subsequence $n_{m_l}$, $l\in\omega$.
Then, for the Peano continuum
$P=\widetilde{\varphi}(\widetilde{t}\times
\mathbb{I}^{\widetilde{B}})$ we have $P\cap
U_{n_{m_l}}\neq\emptyset$ for $l\in\omega$.

By (II), we claim that $M$ is a locally connected space. Assume
the contrary. Then there is a non-empty open set $V\subseteq M$
and some non-open component $K\subseteq V$. Let $t\in K\setminus
Int K$. Since $K$ is closed in $V$, there is a sequence of
non-empty open sets $U_n\subseteq V\setminus K$ and
$U_n\rightarrow t$. By (II), there exists a Peano continuum $P$
such that $P\cap U_n\neq \emptyset$ for $n\in N_1\subseteq \omega$
($N_1$ is infinite). Then $t\in P$ and there is a connected
neighborhood $O(t)\subseteq V$ of $t$ in $P$. It follows that
$O(t)\cap U_{n_0}\neq \emptyset$ for some $U_{n_0}$. Hence $K$ is
not a component of $V$ which is a contradiction.

Since $M$ is locally connected, connected and compact, it is a
Peano continuum.
\end{proof}


\medskip

{\bf Acknowledgements.} The research funding from the Ministry of
Science and Higher Education of the Russian Federation (Ural
Federal University Program of Development within the Priority-2030
Program) is gratefully acknowledged.

 The authors are grateful to the referee for useful remarks and suggestions.

\bibliographystyle{model1a-num-names}
\bibliography{<your-bib-database>}

\end{document}